\documentclass[11pt]{amsart}
\usepackage{latexsym}
\usepackage{amsmath,amssymb,amsfonts}
\usepackage{graphics}
\usepackage[all]{xy}

\def\Bbb{\mathbb}

\newtheorem{thm}{Theorem}[section]

\newenvironment{rmk}{\mbox{ }\\{\bf  Remark}\mbox{ }}{
\hfill $\Box$\mbox{}\bigskip}

\def\blacksquare{\hbox to .60em {\vrule width .60em height .60em}}

\begin{document}
\renewcommand{\theequation}{\thesection.\arabic{equation}}

\title[Liouville-type theorems and applications to geometry]{Liouville-type theorems and applications to geometry on complete Riemannian manifolds}

\author{Chanyoung Sung}
\date{\today}
\address{Dept. of Mathematics and Institute for Mathematical Sciences \\
Konkuk University\\
         1 Hwayang-dong, Gwangjin-gu, Seoul, KOREA}
\email{cysung@kias.re.kr}
\thanks{This work was supported by the National Research Foundation(NRF) grant funded by the Korea government(MEST). (No. 2009-0074404)}



\begin{abstract}
On a complete Riemannian manifold $M$ with Ricci curvature
satisfying $$\textrm{Ric}(\nabla r,\nabla r) \geq -Ar^2(\log
r)^2(\log(\log r))^2\cdots (\log^{k}r)^2$$ for $r\gg 1$, where
$A>0$ is a constant, and $r$ is the distance from an arbitrarily
fixed point in $M$. we prove some Liouville-type theorems for a
$C^2$ function $f:M\rightarrow \Bbb R$ satisfying $\Delta f\geq
F(f)$ for a function $F:\Bbb R\rightarrow \Bbb R$.

As an application, we obtain a $C^0$ estimate of a spinor
satisfying the Seiberg-Witten equations on such a manifold of
dimension 4. We also give applications to the Yamabe problem and isometric immersions of such a manifold.
\end{abstract}
\maketitle
\setcounter{section}{0}
\setcounter{equation}{0}

\section{Introduction}

According to Liouville's theorem, any $f\in C^2(\Bbb R^2)$ which is subharmonic ($\Delta f\geq 0$)
and bounded above must be constant. This is not true in higher dimensions,
but various types of extensions to general complete Riemannian manifolds have been found.
We are here concerned with the case $\Delta f\geq F(f)$, where $\Delta$ is the Laplace-Beltrami operator and $F$ is a real-valued function on $\Bbb R$.

L. Karp \cite{karp} showed that on a complete Riemannian manifold
$M$ satisfying $$\limsup_{r\rightarrow \infty}\frac{\log
\textrm{Vol}\ B(p,r)}{r^2} < \infty$$ for a point $p\in M$ where
$B(p,r)$ is the geodesic ball centered at $p$ with radius $r$,
there exists no $f\in C^2(M)$ which is strongly subharmonic
($\Delta f\geq c$ for a constant $c>0$) and bounded above. (This
volume growth condition holds when the Ricci curvature satisfies
$$\textrm{Ric} \geq -A(1+r^2)$$ where $A>0$ is a constant and $r$ is the distance from $p$.)

S.M. Choi, J.H. Kwon, and Y.J. Suh \cite{suh} proved that on a complete Riemannian manifold with Ricci curvature bounded below, every nonnegative $C^2$ function $f$ satisfying $\Delta f \geq c f^d$ for constants $c>0$ and $d>1$ must vanish identically.

In this article, we consider the general type $F(f)$ to prove :
\begin{thm}\label{th1}
Let $M$ be a smooth complete Riemannian manifold with Ricci curvature satisfying $$\textrm{Ric}(\nabla r,\nabla r) \geq -Ar^2(\log r)^2(\log(\log r))^2\cdots (\log^{k}r)^2$$ for $r\gg 1$, where $A>0$ is a constant, $r$ is the distance from an arbitrarily fixed point in $M$, and $\log^{k}$ denotes the $k$-fold composition of $log$.

Suppose that a $C^2$ function $f:M\rightarrow \Bbb R$ is bounded below and satisfies $\Delta f\geq F(f)$ for a real-valued function $F$ on $\Bbb R$.
\begin{itemize}
\item If $\liminf_{x\rightarrow \infty}
    \frac{F(x)}{x^{\nu}}>0$ for some $\nu>1$, then $f$ is
    bounded such that $F(\sup f)\leq 0$.
\item If $\liminf_{x\rightarrow \infty}
    \frac{F(x)}{x^{\nu}}\leq 0$ for any $\nu>1$, then $\sup
    f=\infty$  or $f$ is bounded such that $F(\sup f)\leq 0$.
\end{itemize}
\end{thm}
The proof is based on a generalized Omori-Yau maximum principle which holds under the above Ricci curvature condition.

In section 3, we give a similar result for bounded-above $f$, which leads to an improvement of L. Karp's theorem \cite{karp} :
\begin{thm}\label{th1.5}
Let $M$ be as in theorem \ref{th1}.

Then there exists no $f\in C^2(M)$ which
is strongly subharmonic and bounded above, and any  $f\in C^2(M)$ which is nonpositive and satisfies $\Delta f\geq
c|f|^d$ for some positive constants $c$ and $d$ must be identically zero.
\end{thm}

In later sections, we discuss the geometric application of these theorems on manifolds of theorem \ref{th1}. We derive a $C^0$ estimate of the spinor satisfying the Seiberg-Witten equations on 4-manifolds, and we give slight improvements on well-known applications such as the Yamabe problem and isometric immersions of such a manifold by using our Liouville-type theorems.

Finally we remark that in all the theorems of this article the Ricci curvature assumption can be replaced by a weaker condition that $M$ admits a tamed-exhaustion, which guarantees the Omori-Yau maximum principle by K.-T. Kim and H. Lee \cite{hanjin}.

\section{Proof of theorem \ref{th1}}
We follow the idea of \cite{suh}. Since $f$ is bounded below, we
take a constant $a$ bigger than $\inf f$. Define a $C^2$ function
$G : M\rightarrow \Bbb R^+$ by  $(f+a)^{\frac{1-q}{2}}$ where
$q>1$ is a constant. By a simple
computation,
\begin{eqnarray}\label{form1}
\nabla G = \frac{1-q}{2}G^{\frac{q+1}{q-1}}\nabla f,
\end{eqnarray}
 and
\begin{eqnarray} \label{form2}
\Delta G=-\frac{q+1}{2}G^{\frac{2}{q-1}}\nabla G\cdot \nabla f + \frac{1-q}{2}G^{\frac{q+1}{q-1}}\Delta f.
\end{eqnarray}
By Plugging (\ref{form1}) to (\ref{form2}), we get
\begin{eqnarray} \label{form3}
\frac{1-q}{2}G^{\frac{2q}{q-1}}\Delta f=G\Delta G-\frac{q+1}{q-1}|\nabla G|^2.
\end{eqnarray}

Since $G$ is bounded below, we can apply the generalized  Omori-Yau maximum principle due to A. Ratto, M. Rigoli, and A.G. Setti \cite{rrs} to have that
for any $\epsilon >0$ there exists a point $p\in M$ such that
$$|\nabla G(p)| < \epsilon, \ \ \ \Delta G(p) > - \epsilon, \ \ \ \textrm{and} \ \ \ \inf G +\epsilon > G(p).$$
We take a sequence $\epsilon_m >0$ converging to 0 as $m\rightarrow \infty$, and get corresponding $p_m\in M$. Observe that
$G(p_m)\rightarrow \inf G$, and hence $f(p_m)\rightarrow \sup f$.

Evaluating (\ref{form3}) at $p_m$ gives
\begin{eqnarray}\label{form4}
\frac{1-q}{2}G(p_m)^{\frac{2q}{q-1}}\Delta f(p_m)> -\epsilon_m G(p_m)-\frac{q+1}{q-1}\epsilon_m^2.
\end{eqnarray}
Applying $\Delta f\geq F(f)$ and replacing $G$ by $(f+a)^{\frac{1-q}{2}}$, we get
\begin{eqnarray}\label{form5}
\frac{F(f(p_m))}{(f(p_m)+a)^q}< \frac{2\epsilon_m}{q-1}\frac{1}{(f(p_m)+a)^{\frac{q-1}{2}}} +\frac{2(q+1)}{(q-1)^2}\epsilon_m^2.
\end{eqnarray}

Suppose $\sup f < \infty$. Then as $m\rightarrow \infty$, the RHS converges to 0 while the LHS converges to $\frac{F(\sup f)}{(\sup f+a)^q} $. Thus we can conclude that $F(\sup f)\leq 0$.

Now it only remain to show that when $\liminf_{x\rightarrow
\infty} \frac{F(x)}{x^{\nu}}>0$ for some $\nu>1$, $f$ must be
bounded. Let's assume $\sup f=\infty$. Then for $q< \nu$, the
LHS diverges to $\infty$, while the RHS converges to 0  as
$m\rightarrow \infty$. This is a contradiction completing the
proof.

\begin{rmk}
In the second case, unbounded examples exist a lot. For instance, on  $\Bbb R^n$ with the Euclidean metric, consider
$$f(x_1,\cdots,x_n)=e^{x_1}+\cdots +e^{x_n}$$ or $$f(x_1,\cdots,x_n)=e^{x_1+\cdots+x_n}.$$ Then $f$ is bounded below but not bounded above, while $\Delta f=f$. This answers the question raised in \cite{suh}.
\end{rmk}

\begin{rmk}
Note that Ricci curvature condition is needed only for the application of the Omori-Yau maximum principle. According to the result of A. Ratto, M. Rigoli, and A.G. Setti \cite{rrs}, the maximum principle holds if
 $$Ric(\nabla r, \nabla r)\geq -B^2\rho(r)$$
for some constant $B>0$, and some smooth nondecreasing function $\rho(r)$ on $[0,\infty)$ which satisfies $$ \rho(0)=1,\ \ \ \ \rho^{(2k+1)}(0)=0 \ \ \forall k\geq 0,$$  $$\sqrt{\rho(t)}\notin L^1,\ \ \ \ \limsup_{t\rightarrow \infty} \frac{t\rho(\sqrt{t})}{\rho(t)} < \infty. $$

 Moreover K.-T. Kim and H. Lee \cite{hanjin} found that the Omori-Yau maximum principle holds in a weaker condition that $M$ admits a tamed-exhaustion, i.e.
a nonnegative continuous function $u: M\rightarrow \Bbb R$
such that $\{p:u(p)\leq r\}$ is compact for every $r\in \Bbb R$, and at every $p\in M$ there exists a $C^2$ function $v$ defined on an open neighborhood of $p$ satisfying $$v\geq u,\ \ v(p)=u(p), \ \ |\nabla v(p)|\leq 1, \ \ \Delta v(p)\leq 1.$$
\end{rmk}

\section{Proof of theorem \ref{th1.5}}
\setcounter{equation}{0}

The results follow from
\begin{thm}\label{th2}
Let $M$ be as in theorem \ref{th1.5}. For a $C^2$ function
$f:M\rightarrow \Bbb R$ which is bounded above and satisfies
$\Delta f\geq F(f)$ where $F:\Bbb R\rightarrow \Bbb R$,
\begin{itemize}
\item if $\liminf_{x\rightarrow -\infty} \frac{F(x)}{(-x)^{\nu}}>0$ for some $\nu<1$, then $f$ is bounded such that $F(\inf f)\leq 0$.
\item if $\liminf_{x\rightarrow -\infty} \frac{F(x)}{(-x)^{\nu}}\leq 0$ for any $\nu<1$, then $\inf f=-\infty$  or $f$ is bounded such that $F(\inf f)\leq 0$.
\end{itemize}
\end{thm}
\begin{proof}
The method is similar to theorem \ref{th1}.
Since $-f$ is bounded below, we apply the proof of theorem \ref{th1} to $-f$ with $q<1$ to get
\begin{eqnarray*}
\frac{1-q}{2}G(p_m)^{\frac{2q}{q-1}}\Delta (-f)(p_m)> -\epsilon_m G(p_m)-\frac{|q+1|}{|q-1|}\epsilon_m^2
\end{eqnarray*}
from (\ref{form4}), and hence
\begin{eqnarray*}
\frac{\Delta f(p_m)}{(-f(p_m)+a)^q}&<& \frac{2\epsilon_m}{1-q}\frac{1}{(-f(p_m)+a)^{\frac{q-1}{2}}} +\frac{2|q+1|}{(q-1)^2}\epsilon_m^2.
\end{eqnarray*}
Applying $\Delta f\geq F(f)$ and simplifying, we have
$$\frac{F(f(p_m))}{(-f(p_m)+a)^{\frac{q+1}{2}}}< \frac{2\epsilon_m}{1-q}+\frac{2|q+1|}{(q-1)^2}\epsilon_m^2(-f(p_m)+a)^{\frac{q-1}{2}}.$$

Now if $\inf f > -\infty$, then we get $F(\inf f)\leq 0$ by letting $m\rightarrow \infty$.

In case $\liminf_{x\rightarrow -\infty} \frac{F(x)}{(-x)^{\nu}}>0$
for some $\nu<1$, to show the boundedness of $f$ let's assume
$\inf f=-\infty$. Then taking $q$ such that $\frac{q+1}{2}< \nu$
and letting $m\rightarrow \infty$, the LHS diverges to $\infty$,
while the RHS converges to zero. From this contradiction, we
conclude that $f$ must be bounded, completing the proof.
\end{proof}

If $f$ is bounded-above and satisfies $\Delta f\geq c$ for a constant $c>0$, applying the above theorem with $F=c$, it follows that $f$ is bounded and $F(\inf f)\leq 0$. This is contradictory to $F\equiv c >0$.

For the second one, applying the above theorem with $F(f)=c|f|^d$, one can conclude that $f$ is bounded and $c|\inf f|^d\leq 0$ implying $f\equiv 0$.

\begin{rmk}
P.-F. Leung \cite{Leung} showed that  on a complete noncompact Riemannian manifold $M$ with $\lambda_1(M)=0$, there exists no bounded  $C^2$ function $f$ satisfying $\Delta f \geq c$ for a positive constant $c$. Here $$\lambda_1(M):=\lim_{r\rightarrow \infty}\lambda_1(B(p,r))$$ for any $p\in M$, where $\lambda_1(B(p,r))$ is the Dirichlet eigenvalue of $\Delta$ in $B(p,r)$, and it is known that $\lambda_1(M)=0$ if the Ricci curvature of $M$ is nonnegative.
\end{rmk}

\section{Application to Seiberg-Witten equations}

We now use theorem \ref{th1} to derive an upper bound of a solution of the Seiberg-Witten equations of which we give here a brief account.
Let $M$ be a smooth oriented Riemannian $4$-manifold.
Consider oriented $\Bbb R^3$-vector bundles $\wedge^2_+$ and
$\wedge^2_-$ consisting of self-dual $2$-forms and
anti-self-dual $2$-forms respectively. Let's let $P_1$ and $P_2$
be associated $SO(3)$ frame bundles. Unless $M$ is spin, it is
impossible to lift these to principal $SU(2)$ bundles. Instead
there always exists the $\Bbb Z_2$-lift, a principal
$U(2)=SU(2)\otimes_{\Bbb Z_2} U(1)$ bundle, of a $SO(3)\oplus
U(1)$ bundle, when the $U(1)$ bundle on the bottom, denoted by $L$,
has first Chern class equal to $w_2(TM)$ modulo $2$.  We call this
lifting a Spin$^c$ structure on $M$.

Let $W_+$ and $W_-$ be $\Bbb C^2$-vector bundles associated to
the above-obtained principal $U(2)$ bundles. One can define a
connection $\nabla_A$ on them by lifting the Levi-Civita
connection and a $U(1)$ connection $A$ on $L$. Then the Dirac
operator $D_A : \Gamma(W_+)\rightarrow \Gamma(W_-)$ is defined as
the composition of $\nabla_A:\Gamma(W_+)\rightarrow T^*M\otimes
\Gamma(W_+)$ and the Clifford multiplication.

For a section $\Phi$ of $W_+$, the Seiberg-Witten equations
of $(A,\Phi)$ is given by
$$\left\{
\begin{array}{ll} D_A\Phi=0\\
  F_{A}^+= \Phi\otimes\Phi^*-\frac{|\Phi|^2}{2}\textrm{Id},
\end{array}\right.
$$
where $F_A^+$ is the self-dual part of the curvature $dA$ of $A$,  and  the identification of both sides in the second equation comes from the Clifford action.

It is essential and the first step to obtain a $C^0$ bound on the spinor part of a solution in order to obtain the compactness of its moduli space of solutions. It is also essential for proving the emptiness of the moduli space when the Riemannian metric of $M$ has positive scalar curvature.
When $M$ is compact, such a bound can be easily derived, because there exists a point in $M$ where the maximum norm is attained.
When $M$ is noncompact, one cannot guarantee such a bound in general, but we prove :
\begin{thm}
Let $M$ be a smooth oriented complete Riemannian 4-manifold with the Ricci curvature condition as in theorem \ref{th1}. Suppose $(A,\Phi)$ is a solution of the Seiberg-Witten equations for a Spin$^c$ structure on $M$. Then $$\sup |\Phi|^2 \leq \sup s^-,$$
where $s^-(x):=\max(-s(x),0)$ and $s: M\rightarrow \Bbb R$ is the scalar curvature.
\end{thm}
\begin{proof}
We may assume $\sup s^-< \infty$, otherwise there is nothing to prove.
Recall the Weitzenb\"ock formula
$$D_AD_A\Phi=\nabla_A^*\nabla_A \Phi+\frac{s}{4}\Phi+\frac{F_A^+}{2}\cdot\Phi.$$
For a solution $(A,\Phi)$,
$$0=\nabla_A^*\nabla_A \Phi+\frac{s}{4}\Phi+\frac{|\Phi|^2}{4}\Phi.$$
Taking the inner product with $\Phi$ gives
$$0=-\frac{1}{2}\Delta |\Phi|^2+|\nabla\Phi|^2+\frac{s}{4}|\Phi|^2+\frac{1}{4}|\Phi|^4,$$ and hence we get
$$\Delta |\Phi|^2\geq -\frac{\sup s^-}{2}|\Phi|^2+\frac{1}{2}|\Phi|^4.$$
Now we apply theorem \ref{th1} with $f=|\Phi|^2$ and $F(f)=-\frac{\sup s^-}{2}f+\frac{1}{2}f^2$, and obtain
$$ -\frac{\sup s^-}{2} \sup |\Phi|^2 +  \frac{1}{2}(\sup |\Phi|^2)^2 \leq 0,$$ which implies $$\sup |\Phi|^2 \leq  \sup s^-.$$
\end{proof}

\begin{rmk}
One can also derive the corresponding estimate for perturbed
Seiberg-Witten equations in the same way.


It is interesting to ask whether this estimate still holds without the Ricci curvature condition.
\end{rmk}

\section{Application to the Yamabe problem}

\begin{thm}\label{yam}
Let $(M,g)$ be as in theorem \ref{th1}. Suppose that the scalar
curvature  of $g$ is nonnegative. If $\tilde{s}(x)$ is a smooth
nonpositive function on $M$ such that
$$\sup_{x\in M-K} \tilde{s}(x)<0$$ for a compact subset
$K\subset M$, then $g$ cannot be conformally deformed to a metric
of scalar curvature $\tilde{s}$.
\end{thm}
\begin{proof}
We may let $n:=\dim M \geq 2$, otherwise there is nothing to
prove.

First, let's consider the case when $n\geq 3$. Suppose there
exists a smooth positive function $f$ on $M$ such that the scalar
curvature of $f^{\frac{4}{n-2}}g$ is $\tilde{s}$. Then
$$4\frac{n-1}{n-2}\Delta f-sf+\tilde{s}f^{\frac{n+2}{n-2}}=0,$$
where $s$ is the scalar curvature of $g$. Obviously $f$ is not
constant. Applying the maximum principle to the inequality
$$4\frac{n-1}{n-2}\Delta f-sf\geq 0,$$ we have that the maximum
does not occur on $M$.

Now we apply theorem \ref{th1} to
$$4\frac{n-1}{n-2}\Delta f\geq -\tilde{s} f^{\frac{n+2}{n-2}}.$$
Although $F$ is not only a function of $f$ but also $x\in M$, the
proof of theorem \ref{th1} works well all the way to give
(\ref{form5}). Since $f(p_m)\rightarrow \sup f$, $p_m\in M-K$ for
sufficiently large $m$. Then in (\ref{form5}) replacing
$F(f(p_m))$ with $c f^{\frac{n+2}{n-2}}$ where $c=-\sup_{x\in M-K}
\tilde{s}(x) $, and proceeding the argument as in theorem
\ref{th1}, we conclude that $f$ is bounded and $$c (\sup
f)^{\frac{n+2}{n-2}}\leq 0,$$ meaning $\sup f=0.$ This yields a
desired contradiction.

When $n=2$, the proof is almost the same. Assume to contrary that
there exists a smooth positive function $f$ on $M$ such that the
$fg$ has scalar curvature $\tilde{s}$. The scalar curvature
equation is now $$\Delta \ln f-s+\tilde{s}f=0.$$ To convert it
into the form applicable to theorem \ref{th1}, consider
\begin{eqnarray*}
\Delta f &=& \Delta e^{\ln f}\\ &=& e^{\ln f}|\nabla \ln f|^2+e^{\ln
f}\Delta \ln f\\
&=& f|\nabla \ln f|^2+ sf-\tilde{s}f^2.
\end{eqnarray*}
Then the maximum of $f$ is not attained on $M$ by applying the
maximum principle to $\Delta f \geq 0$, and the application of
theorem \ref{th1} to $\Delta f\geq - \tilde{s}f^2$ as above yields
$\sup f=0$, which is a contradiction.
\end{proof}

\begin{rmk}
It is worth mentioning that on a complete Riemannian manifold
whose scalar curvature $\tilde{s}$ satisfies
$$\tilde{s}\leq -c$$ for a constant $c>0$ outside a compact
subset, there is conformal complete metric with scalar curvature
$\equiv -1$. (For a proof, see \cite{McOwen}.)

When the Ricci curvature of $M$ satisfies  sharper estimates,
better results  hold as obtained by \cite{rrs}. For example, if
$$\textrm{Ric}(\nabla r,\nabla r) \geq -A(1+r^2),$$ then the same conclusion also holds for $\tilde{s}$
such that $$\tilde{s} \leq - \frac{C}{\log r(\log(\log r))\cdots
\log^{k}r},\ \ \ \ r\gg 1$$ for a constant $C>0$.
\end{rmk}

\section{Application to isometric immersions}
We can give a slight generalization of L. Karp's result \cite{karp}.
\begin{thm}
Let $M$ be a smooth complete Riemannian manifold with scalar curvature $s$ satisfying $$s \geq -Ar^2(\log r)^2(\log(\log r))^2\cdots (\log^{k}r)^2$$ for $r\gg 1$, where $A>0$ is a constant, $r$ is the distance from an arbitrarily fixed point in $M$, and $\log^{k}$ denotes the $k$-fold composition of $log$.

Suppose $M$ is isometrically immersed in a geodesic ball of radius $R$ in a simply-connected complete Riemannian manifold $N$ with higher dimension and nonpositive curvature. Then $$R\geq \frac{1}{H_0}$$ where $H_0=\sup ||H||$ and $H$ is the mean curvature. In particular $M$ cannot be minimally immersed in a bounded subset of $N$.
\end{thm}
\begin{proof}
We may assume $\sup ||H|| < \infty$, otherwise there is nothing to prove.
The method of proof is almost the same as \cite{karp}, and we briefly sketch the proof. First, we need to have our desired Ricci curvature estimate under the given situation.
By the Gauss curvature equation, the sectional curvatures $K_M$ and $K_N$ of $M$ and $N$ respectively are related by
\begin{eqnarray*}
K_M(E_1\wedge E_2) &\geq& K_N(E_1\wedge E_2)-2||\alpha||^2\\
&=& K_N(E_1\wedge E_2)-2n^2 ||H||^2+2s-2\sum_{i\ne j}K_N(E_i\wedge E_j)
\end{eqnarray*}
where $\alpha$ is the second fundamental form, and $\{E_i\}$ is a
local orthonormal frame of $M$. Using that $K_N$ is nonpositive
and $\sup ||H||< \infty$, we can conclude that the sectional
curvature and hence the Ricci curvature of $M$ is bounded below by
$$-A'r^2(\log r)^2(\log(\log r))^2\cdots (\log^{k}r)^2$$ for some constant $A'>0$, for sufficiently large $r$.

Let's let $M$ be contained in a ball of radius $R$ and center
$x_0\in N$. As shown in \cite{karp}, for $f(x)=r_N^2(x)\in
C^\infty(N)$ where $r_N$ is the distance from $x_0$ measured in
$N$
\begin{eqnarray*}
\Delta_Mf &=& tr_M(\nabla^2_Nf)+n\langle H, \nabla_N f\rangle_N\\
&\geq & tr_M(\nabla^2_Nf)-nH_0\cdot 2r_N\\
&\geq& 2n-2nH_0R,
\end{eqnarray*}
where $n$ is $\dim M$, $\nabla_N$ is the covariant derivative in
$N$, and $\nabla^2_Nr_N^2\geq 2\langle\cdot,\cdot\rangle_N$ is due
to the Hessian comparison theorem between $N$ and the Euclidean
space.

Now we apply theorem \ref{th1} to $\Delta_Mf\geq 2n-2nH_0R$ to get
the desired inequality $2n-2nH_0R\leq 0$. In case $H_0=0$, we have
$\Delta_Mf \geq 2n$, which implies that $f$ must be unbounded by
theorem \ref{th1.5}. Therefore $R$ cannot be finite. This
completes the proof.
\end{proof}

\bigskip




\begin{thebibliography}{99}

\bibitem{McOwen} P. Aviles and R. McOwen, {\em Conformal deformation to constant negative scalar curvature on noncompact Riemannian manifolds}, J. Diff. Geom.
 {\bf 27} (1988), 225--239.


\bibitem{suh} S. M. Choi, J.-H. Kwon, and Y. J. Suh, {\em A Liouville-type theorem for complete Riemannian manifolds}, Bull. Korean Math. Soc.
 {\bf 35} (1998), 301--309.


\bibitem{hanjin} K.-T. Kim  and H. Lee, {\em On the Omori-Yau almost maximum principle}, J. Math. Anal. App. {\bf 335} (2007), 332--340.

\bibitem{karp} L. Karp, {\em Differential inequalities on complete Riemannian manifolds and applications}, Math. Ann. {\bf 272} (1985), 449--459.

\bibitem{Leung} P.-F. Leung, {\em A Liouville-type theorem for strongly subharmonic functions on complete non-compact Riemannian manifolds and some applications}, Geom. Dedicata {\bf 66} (1997), 159--162.

\bibitem{omori} H. Omori,
 {\em  Isometric immersions of Riemannian manifolds}, J. Math. Soc. Japan {\bf 19} (1967), 205--211.

\bibitem{rrs} A. Ratto, M. Rigoli, and A.G. Setti, {\em On the Omori-Yau maximum principle  and its application to differntial equations and geometry}, J. Func. Anal. {\bf 134} (1995), 486--510.

\bibitem{Yau} S. T. Yau, {\em  Harmonic functions on complete Riemannian manifolds}, Comm. Pure and Appl. Math. {\bf 28} (1975), 201--228.


\end{thebibliography}
\end{document}